\documentclass[12pt]{amsart}

\usepackage{amssymb}

\newtheorem{theorem}{Theorem}[section]
\newtheorem{proposition}[theorem]{Proposition}
\newtheorem{lemma}[theorem]{Lemma}

\newtheorem{definition}[theorem]{Definition}
\newtheorem{remark}[theorem]{Remark}

\def\P{{\mathbb P}}
\def\C{{\mathbb C}}

\def\Z{{\mathbb Z}}

\def\h{\mathfrak{h}}

\numberwithin{equation}{section}

\begin{document}

\baselineskip=15pt

\title[Quantization of moduli of parabolic bundles
on ${\mathbb C}{\mathbb P}^1$]{Quantization of some moduli spaces
of parabolic vector bundles on ${\mathbb C}{\mathbb P}^1$}

\author[I. Biswas]{Indranil Biswas}

\address{School of Mathematics, Tata Institute of Fundamental
Research, Homi Bhabha Road, Bombay 400005, India}

\email{indranil@math.tifr.res.in}

\author[C. Florentino]{Carlos Florentino}

\address{Departament of Mathematics, Center for Mathematical Analysis, Geometry and Dynamical Systems, 
Instituto Superior
T\'ecnico, Av. Rovisco Pais, 1049-001 Lisbon, Portugal}

\email{cfloren@math.ist.utl.pt}

\author[J. Mour\~ao]{Jos\'e Mour\~ao}

\address{Departament of Mathematics, 
Center for Mathematical Analysis, Geometry and Dynamical Systems, Instituto Superior
T\'ecnico, Av. Rovisco Pais, 1049-001 Lisbon, Portugal}

\email{jmourao@math.ist.utl.pt}

\author[J. P. Nunes]{Jo\~ao P. Nunes}

\address{Departament of Mathematics, 
Center for Mathematical Analysis, Geometry and Dynamical Systems, Instituto Superior
T\'ecnico, Av. Rovisco Pais, 1049-001 Lisbon, Portugal}

\email{jpnunes@math.ist.utl.pt}



\date{}

\begin{abstract}
We address quantization of the natural symplectic structure on
a moduli space of parabolic vector bundles of
parabolic degree zero over ${\mathbb C}{\mathbb P}^1$ with four
parabolic points and parabolic weights in $\{0\, ,1/2\}$.
Identifying such parabolic bundles as vector bundles on an
elliptic curve, we obtain explicit expressions for the
corresponding non-abelian theta functions. 
These non-abelian theta functions are
described in terms of certain naturally defined distributions 
on the compact group $\mathrm{SU(2)}$.
\end{abstract}

\maketitle

\tableofcontents

\section{Introduction}\label{sec1}

Let $X$ be a compact connected
Riemann surface, or equivalently a smooth complex
projective curve. It is well known that the moduli spaces of vector
bundles over $X$ have a canonical symplectic structure \cite{G},
with integral symplectic form. Indeed, being naturally identified
with spaces of flat connections on a
compact oriented surface, these are important
classical phase spaces of Chern-Simons theory. The natural question
of their quantization was addressed in many articles
\cite{Hi}, \cite{AdPW}.

The geometric quantization of moduli spaces $\mathcal{N}$ of vector
bundles over $X$ in a so-called K\"{a}hler polarization leads to what
is known as spaces of non-abelian theta functions.
More concretely, the K\"{a}hler polarized Hilbert spaces, at level $k=1,2,
\cdots $, are the spaces $H^{0}(\mathcal{N},\, 
\mathcal{L}^{k})$, where $\mathcal{L}$ is a determinant line bundle,
endowed with a natural Chern connection, whose curvature coincides
with the symplectic form. A projectively flat connection was constructed
by Hitchin on the space of complex structures on $\mathcal{N}$ \cite{Hi}, 
providing
a way of identifying different choices of K\"{a}hler polarized Hilbert
spaces.

However, an explicit identification between K\"{a}hler polarized quantizations
and real polarized ones has only been found in a few examples, notably
the case when $X$ is an elliptic curve \cite{AdPW,FMN2}, using the
relationship between the moduli spaces in this case and a certain class
of abelian varieties. In turn, a comparison between real and K\"{a}hler
quantizations for abelian varieties was obtained using a coherent
state transform \cite{FMN2,FMN1,BMN}. 

In this article, we follow the analogous geometric quantization program
for the moduli space $\mathcal{M}_{P}(r)$ of parabolic bundles of
rank $r$ over $\mathbb{C}\mathbb{P}^{1}$ with four parabolic points
and parabolic weights in $\{0\, ,1/2\}$ with
parabolic degree zero \cite{MS,MY}.
It is known that, as in the case of vector bundles, there is a determinant
line bundle $\zeta_{P}$ over the moduli space of parabolic bundles
endowed with a natural Chern connection, whose curvature is the (generally
singular) K\"{a}hler form.

Let $X$ be the elliptic curve which has a degree two map to
$\mathbb{C}\mathbb{P}^{1}$ ramified over the parabolic points.
Using the description of the parabolic bundles of
above type as holomorphic vector bundles
over $X$ equipped with a lift of the involution corresponding
to the degree two covering \cite{Bi1}, we see that, for a given choice
of parabolic structures on these 4 points, $\mathcal{M}_{P}(r)$ has
dimension $d\leq r/2$ and we have a canonical isomorphism
\[
\mathcal{M}_{P}(r)\,\cong \,X^{d}/\Gamma_{d}\, \cong\,
\mathbb{C}\mathbb{P}^{d}\, ,
\]
where $\Gamma_{d}$ is the semi-direct product $(\mathbb{Z}/2\mathbb{Z})^{d}\rtimes \Sigma_d$ 
for the natural action on $(\mathbb{Z}/2\mathbb{Z})^{d}$ of the
symmetric group $\Sigma_d$ for $d$ elements. Moreover, for the natural
polarization line bundle $L$ on the abelian variety $X^{d}$ (associated to
a K\"{a}hler form of area one on $X$), we
obtain an isomorphism $\phi^{*}\zeta_{P}\,\cong\, L^{2}$,
where $\zeta_{P}\,\longrightarrow\,\mathcal{M}_{P}(r)$ 
is the determinant line bundle
and $\phi\,:\,X^{d}\,\longrightarrow\,\mathcal{M}_{P}(r)$ is the
natural quotient (see Sections 2 and 3).

This very concrete description allows the expression of the quantization
Hilbert space at level $k$, namely $H^{0}(\mathcal{M}_{P},\,
\zeta_{P}^{k})$, in terms of the (abelian)
theta functions of level $2k$ on $X^{d}$, and the comparison of
real and K\"{a}hler polarized Hilbert spaces. For this, we need to apply
the framework of \cite{FMN2} for non-abelian theta functions over
the moduli space of rank 2 vector bundles, with trivial
determinant, over $X$ (see Section 4). 
The so-called coherent state transform for Lie groups \cite{Ha}, is an analytic tool which, 
given an invariant Laplacian on a compact Lie group $K$, associates 
holomorphic functions on the complexification $K_\C$ to square integrable functions on $K$. This set up 
can be extended to appropriate spaces of distributions on $K$
\cite{FMN1}. Non-abelian theta functions of level $k$ on ${\mathcal 
M}_P$ are then described in terms of $Ad$-invariant
holomorphic functions on the group ${\rm SL}(2,\mathbb{C})$ with special 
quasi-periodicity properties. These holomorphic functions are obtained from elements in a
vector space of distributions on the compact real form ${\rm SU}(2)$ by
applying the coherent state transform, for time $1/(k+2)$ (see Theorem \ref{main}).

\section{A moduli space of parabolic vector bundles over
${\mathbb C}{\mathbb P}^1$}

Fix a point $p_0\, \in\, {\mathbb C}{\mathbb P}^1 \setminus \{0\, ,1
\, ,\infty\}$. Consider the divisor
\begin{equation}\label{e0}
S\, :=\, \{0\, ,1 \, ,\infty\, , p_0\}\, \subset\, {\mathbb C}
{\mathbb P}^1\, .
\end{equation}
Let
\begin{equation}\label{e1}
f\, :\, X\, \longrightarrow\, {\mathbb C}{\mathbb P}^1
\end{equation}
be the unique double cover ramified exactly over $D$. Therefore, 
$X$ is a complex elliptic curve. Let ${\rm Pic}^0(X)$ be the moduli
space of topologically trivial holomorphic vector bundles on $X$.

\begin{lemma}\label{lem1}
Any polystable vector bundle $E$ over $X$ of rank $r$ and
degree zero is isomorphic
to a direct sum $\bigoplus_{i=1}^r L_i$, where $L_i\, \in\,
{\rm Pic}^0(X)$.

The isomorphism classes of
line bundles $L_i$, $1\, \leq\, i\, \leq\, r$, are uniquely
determined by $E$ up to a permutation of $\{1\, ,\cdots\, ,r\}$.
\end{lemma}

\begin{proof}
The first statement
follows immediately from Atiyah's classification of holomorphic
vector bundles on $X$ (see \cite{At2}). This also follows from
the facts that $E$ is given by a representation
of the abelian group $\pi_1(X)$ in ${\rm U}(r)$ \cite{NaSe}.

The uniqueness of $L_i$ up to a permutation of $\{1\, ,\cdots\, ,r\}$
follows immediately from \cite[p. 315, Theorem 2(ii)]{At1}.
\end{proof}

We will consider parabolic vector bundles over ${\mathbb C}{\mathbb 
P}^1$ with $S$ (see \eqref{e0}) as the parabolic divisor. Let $E$ be
a holomorphic vector bundle on ${\mathbb C}{\mathbb P}^1$. A
\textit{quasi--parabolic} structure on $E$ is a filtration
of subspaces
$$
E_{y} \,=:\, F_{y,1} \, \supsetneq\, \cdots\, \supsetneq\,
F_{y,j} \, \supsetneq\, \cdots\,
\supsetneq\, F_{y,a_y}\, \supsetneq\, F_{y,a_y+1} \,=\, 0
$$
over each point $y\, \in\, S$. A \textit{parabolic} structure on
$E$ is a quasi--parabolic structure as above together with real
numbers
\begin{equation}\label{e-1}
0\, \leq\, \alpha_{y,1} \, <\, \cdots\, <\, \alpha_{y,j}
\, <\, \cdots\, <\, \alpha_{y,a_y}\, <\, 1
\end{equation}
associated to the quasi--parabolic flags. (See \cite{MS}, \cite{MY}.)
The numbers $\alpha_{y,j}$
in \eqref{e-1} are called \textit{parabolic weights}.
The \textit{multiplicity} of the parabolic weight $\alpha_{y,j}$
is $\dim_{\mathbb C} F_{y,j}/F_{y,j+1}$.

For notational convenience, a parabolic vector bundle
$(E\, ,\{F_{y,j}\}\, ,\{\alpha_{y,j}\})$ defined as above will also
be denoted by $E_*$. The \textit{parabolic degree} is defined to be
$$
\text{par-deg}(E_*)\, :=\, \text{degree}(E)+\sum_{y\in S}
\sum_{j=1}^{a_y} \alpha_{y,j}\cdot \dim (F_{y,j}/F_{y,j+1})\, .
$$

Fix an integer $r\, \geq\, 2$. For each point $y\, \in\, S$, fix
an integer $m_y\, \in\, [0\, ,r]$. Let
${\mathcal M}_P$ be the moduli space of semistable parabolic vector 
bundles $E_*$ on ${\mathbb C}{\mathbb P}^1$ of rank $r$, with $S$
as the parabolic divisor, such that the
parabolic weights at a parabolic point $y$ are $1/2$ with 
multiplicity
$m_y$ and $0$ with multiplicity $r-m_y$, and
$$
\text{par-deg}(E_*) \, =\, 0\, .
$$
(See \cite{MY} for the construction of ${\mathcal M}_P$.)
The moduli spaces of parabolic bundles are irreducible normal
complex projective varieties. We will see later that the above moduli
space ${\mathcal M}_P$ is smooth. Note that ${\mathcal M}_P$ is
empty if $\sum_{y\in S} m_y$ is an odd integer.

We will assume that $\sum_{y\in S} m_y$ is an even integer.

For any integer $m\, \geq\, 1$, we will construct a finite group
$\Gamma_m$ equipped with an action of it on the Cartesian product
${\rm Pic}^0(X)^m$.

Let $\Sigma_m$ be the group of permutations of $\{1\, ,\cdots\, ,m\}$.
This group acts on the Cartesian product $({\mathbb Z}/2{\mathbb Z}
)^m$ by permuting the factors. So any permutation $\tau\, \in\,
\Sigma_m$ of $\{1\, ,\cdots\, ,m\}$ sends any $(z_1,\cdots,z_m) \, \in\,
({\mathbb Z}/2{\mathbb Z} )^m$ to $(z_{\tau^{-1}(1)},\cdots,z_{\tau^{-1}(m)})$.
Let
\begin{equation}\label{ni3}
\Gamma_m\, :=\, ({\mathbb Z}/2{\mathbb Z})^m\rtimes \Sigma_m
\end{equation}
be the semi-direct product corresponding to this action. So
$\Gamma_m$ fits in a short exact sequence
\begin{equation}\label{esg}
e\,\longrightarrow\, ({\mathbb Z}/2{\mathbb Z})^m
\,\longrightarrow\,\Gamma_m\,\longrightarrow\, \Sigma_m
\,\longrightarrow\, e
\end{equation}
of groups. We will construct a natural action of $\Gamma_m$ on
${\rm Pic}^0(X)^m$.

Consider the action of group ${\mathbb Z}/2{\mathbb Z}$ on ${\rm 
Pic}^0(X)$ defined by the involution $L\,\longmapsto\, L^*$.
Acting coordinate-wise, it
produces an action of $({\mathbb Z}/2{\mathbb Z})^m$ on
${\rm Pic}^0(X)^m$. On the other hand, the permutation group
$\Sigma_m$ acts on ${\rm Pic}^0(X)^m$; as before, the action of any
$\tau\, \in\, \Sigma_m$ sends any $(z_1,\cdots,z_m) \, \in\,
{\rm Pic}^0(X)^m$ to $(z_{\tau^{-1}(1)},\cdots,z_{\tau^{-1}(m)})$. These
two actions together produce an action of $\Gamma_m$ (constructed
in \eqref{ni3}) on ${\rm Pic}^0(X)^m$. Let
\begin{equation}\label{e3}
{\rm Pic}^0(X)^m\, \longrightarrow\, {\rm Pic}^0(X)^m/\Gamma_m
\end{equation}
be the quotient for this action. The quotient
${\rm Pic}^0(X)^m/({\mathbb Z}/2{\mathbb Z})^m$ for the subgroup
in \eqref{esg} is identified with
$({\rm Pic}^0(X)/({\mathbb Z}/2{\mathbb Z}))^m$. Hence
$$
{\rm Pic}^0(X)^m/\Gamma_m\, =\, \text{Sym}^m({\rm Pic}^0(X)/
({\mathbb Z}/2{\mathbb Z}))\, .
$$
Since ${\rm Pic}^0(X)/({\mathbb Z}/2{\mathbb Z})\, =\,
{\mathbb C}{\mathbb P}^1$, we have
$$
{\rm Pic}^0(X)^m/\Gamma_m\, =\,\text{Sym}^m({\mathbb C}{\mathbb P}^1)
\, =\,{\mathbb C}{\mathbb P}^m\, .
$$

Note that the quotient map in \eqref{e3}
factors through the projection
$$
{\rm Pic}^0(X)^m\, \longrightarrow\, \text{Sym}^m({\rm Pic}^0(X))
\,:=\, {\rm Pic}^0(X)^m/\Sigma_m\, .
$$
But the surjective map
$$\text{Sym}^m({\rm Pic}^0(X))
\, \longrightarrow\,{\rm Pic}^0(X)^m/\Gamma_m$$ in general is not
a quotient for a group action because $\Sigma_m$ is not a normal
subgroup of $\Gamma_m$.

\begin{proposition}\label{prop1}
Let $d$ be the (complex) dimension of ${\mathcal M}_P$. Then
$d\, \leq\,
r/2$. If $d\, >\, 0$, then the variety ${\mathcal M}_P$ is
canonically isomorphic to the quotient ${\rm Pic}^0(X)^d/
\Gamma_d$ constructed in \eqref{e3}.
\end{proposition}

\begin{proof}
Let
\begin{equation}\label{sigma}
\sigma\, :\, X\, \longrightarrow\, X
\end{equation}
be the unique nontrivial deck transformation for the covering $f$
in \eqref{e1}.

Let $E_*\, \in\, {\mathcal M}_P$ be a polystable parabolic vector
bundle. It corresponds to a unique holomorphic vector bundle
$V\, \longrightarrow\, X$ equipped with a lift of the involution
$\sigma$ in \eqref{sigma} as an isomorphism of order two
\begin{equation}\label{ni2}
\widetilde{\sigma}\, :\,V\, \longrightarrow\, \sigma^* V
\end{equation}
of vector bundles \cite{Bi1}; this means that $\widetilde{\sigma}$
is a holomorphic isomorphism of vector bundles, and the composition
$$
V\, \stackrel{\widetilde{\sigma}}{\longrightarrow}\, \sigma^* V
\, \stackrel{\sigma^*\widetilde{\sigma}}{\longrightarrow}\, 
\sigma^*\sigma^* V\,=\, V
$$
is the identity map. We have
$$
\text{degree}(V) \, =\, 0
$$
because $\text{par-deg}(E_*) \, =\, 0$ \cite[p. 318, (3.12)]{Bi1}.
The vector bundle $V$ is polystable because $E_*$ is polystable
\cite[pp. 350--351, Theorem 4.3]{BBN}. Therefore, from Lemma
\ref{lem1} we know that
\begin{equation}\label{ni1}
V\, =\, \bigoplus_{i=1}^r L_i\, ,
\end{equation}
where $L_i\, \in\, {\rm Pic}^0(X)$. Recall from Lemma \ref{lem1}
that the line bundles $L_i$ are uniquely determined up to
a permutation.

For any line bundle $L$ on $X$ of degree zero, the line bundle
$L\otimes \sigma^*L$ descends to $ {\mathbb C}{\mathbb P}^1$,
where $\sigma$ is defined in \eqref{sigma}.
Since ${\rm Pic}^0({\mathbb C}{\mathbb P}^1)\,=\,
\{{\mathcal O}_{{\mathbb C}{\mathbb P}^1}\}$, it follows that
\begin{equation}\label{sigma2}
\sigma^* L\, =\, L^*
\end{equation}
for all $L\, \in\, {\rm Pic}^0(X)$.

Since $V$ in \eqref{ni1} is isomorphic to $\sigma^* V$ (see
\eqref{ni2}), using \eqref{sigma2},
\begin{equation}\label{e4}
\bigoplus_{i=1}^r L_i\,=\, \bigoplus_{i=1}^r L^*_i\, .
\end{equation}
Therefore, all vector bundles on $X$ corresponding to points
of ${\mathcal M}_P$ are of the form
\begin{equation}\label{e5}
\bigoplus_{i=1}^a (\xi_i\oplus \xi^*_i) \oplus \bigoplus_{j=1}^{r-2a}
\eta_j\, , 
\end{equation}
where $\eta_j$ are fixed line bundles on $X$ (these line bundles
$\eta_j$ depend on the numbers $m_y$ but are independent of the
point of the moduli space ${\mathcal M}_P$), and
the line bundles $\xi_i$, $1\,\leq\, i\,\leq\, a$, move over
${\rm Pic}^0(X)$. From \eqref{e4} it follows that 
\begin{equation}\label{e6}
\eta_j\, =\, \eta^*_j
\end{equation}
for all $j$. We note that any vector bundle as in \eqref{e5}
satisfying \eqref{e6} admits
a lift of the involution $\sigma$. Indeed,
each $\eta_j$ has a lift because \eqref{e6} holds. Also,
$\xi_i\oplus \xi^*_i$ has a natural lift of the
involution $\sigma$
because $\sigma^*\xi_i\,=\, \xi^*_i$. Note that the
involution of $\xi_i\oplus \xi^*_i$ interchanges the
two direct summands.

Hence we get a surjective morphism
\begin{equation}\label{m}
{\rm Pic}^0(X)^a\, \longrightarrow\, {\mathcal M}_P
\end{equation}
that sends any $(\xi_1\, ,\cdots \, , \xi_a)$ to
$$
\bigoplus_{i=1}^a (\xi_i\oplus \xi^*_i) \oplus \bigoplus_{j=1}^{r-2a}
\eta_j\, .
$$
This morphism clearly factors through the quotient
${\rm Pic}^0(X)^a/\Gamma_a$ in \eqref{e3}.

For a vector bundle
$$
W\,=\, \bigoplus_{i=1}^a (\xi_i\oplus \xi^*_i)\, ,
$$
the unordered pairs $\{\xi_i\, ,\xi^*_i\}$ are uniquely
determined by $W$ up to a permutation of $\{1\, ,\cdots\, ,a\}$
\cite[p. 315, Theorem 2(ii)]{At1}.
Using this it follows that the above morphism
$$
{\rm Pic}^0(X)^a/\Gamma_a\, \longrightarrow\,{\mathcal M}_P
$$
is an isomorphism. This completes the proof of the proposition.
\end{proof}

\section{Determinant line bundle and K\"ahler form on
${\mathcal M}_P$}

Consider the moduli space ${\mathcal M}_P$ of parabolic
vector bundles defined in the previous section. It has a natural
(possibly singular) K\"ahler form; this K\"ahler form will be denoted
by $\omega_P$. There is a determinant line bundle
\begin{equation}\label{zeta}
\zeta\, \longrightarrow\, {\mathcal M}_P\, .
\end{equation}
This line bundle $\zeta$ has a hermitian structure such that the
curvature of the corresponding Chern connection coincides
with $\omega_P$. (See \cite{BR}, \cite{Bi2}, \cite{TZ}.)

Consider the dimension
$d$ in Proposition \ref{prop1}. Let $\Gamma_d$ be the group
defined in \eqref{ni3}. The quotient
${\rm Pic}^0(X)^d/\Gamma_d$ is identified with the moduli space
${\mathcal M}_P$ by Proposition \ref{prop1}. Let
\begin{equation}\label{phi}
\phi\, :\, {\rm Pic}^0(X)^d\, \longrightarrow\,
{\rm Pic}^0(X)^d/\Gamma_d\,=\, {\mathcal M}_P
\end{equation}
be the morphism in \eqref{m}. Since $\phi$ is the quotient
map for the
action of $\Gamma_d$ on ${\rm Pic}^0(X)^d$, the pulled back
line bundle $\phi^*\zeta$ is equipped with a lift of the action
of $\Gamma_d$ on ${\rm Pic}^0(X)^d$, where $\zeta$ is the determinant
line bundle in \eqref{zeta}.

Let
\begin{equation}\label{l0}
L_0\, :=\, {\mathcal O}_{{\rm Pic}^0(X)}({\mathcal O}_X)\,
\longrightarrow\, \text{Pic}^0(X)
\end{equation}
be the holomorphic line bundle of degree one
defined by the point of $\text{Pic}^0(X)$
corresponding to the trivial line bundle ${\mathcal O}_X$ on
$X$. For each $i\, \in\, [1\, ,d]$, let
\begin{equation}\label{qi}
q_i\, :\, {\rm Pic}^0(X)^d\, \longrightarrow\, {\rm Pic}^0(X)
\end{equation}
be the projection to the $i$--th factor. The action of
$\Sigma_d$ on ${\rm Pic}^0(X)^d$ that permutes the factors
in the Cartesian product has a natural lift to an action of
$\Sigma_d$ on the line bundle
$$
\bigotimes_{i=1}^d q^*_i L_0\, \longrightarrow\, {\rm Pic}^0(X)^d\, ,
$$
where $L_0$ is the line bundle in \eqref{l0}.

Recall that the group $({\mathbb Z}/2{\mathbb Z})^d$ acts on
${\rm Pic}^0(X)^d$ using the action of ${\mathbb Z}/2{\mathbb Z}$
on ${\rm Pic}^0(X)$ given by the involution $L\, \longrightarrow\,
L^*$. Let
\begin{equation}\label{in4}
{\sigma}_d\, :\, ({\mathbb Z}/2{\mathbb Z})^d\,
\longrightarrow\, \text{Aut}({\rm Pic}^0(X)^d)
\end{equation}
be the corresponding homomorphism.

\begin{theorem}\label{thm1}
For any $g\, \in\, ({\mathbb Z}/2{\mathbb Z})^d$,
there is a canonical isomorphism of holomorphic line bundles
$$
{\sigma}_d(g)^*(\bigotimes_{i=1}^d q^*_i L_0)\,
\stackrel{\sim}{\longrightarrow}\,\bigotimes_{i=1}^d q^*_i L_0\, ,
$$
where ${\sigma}_d$ is the homomorphism in \eqref{in4}, and $L_0$
is the line bundle in \eqref{l0}.

The line bundle $(\bigotimes_{i=1}^d q^*_i L_0)^{\otimes 2}$ has a
canonical lift of the action of $\Gamma_d$ on ${\rm Pic}^0(X)^d$.

There is a $\Gamma_d$--equivariant isomorphism of line bundles
$$
(\bigotimes_{i=1}^d q^*_i L_0)^{\otimes 2}\, 
\stackrel{\sim}{\longrightarrow}\, \phi^*\zeta\, ,
$$
where $\phi$ and $\zeta$ are defined in \eqref{phi} and
\eqref{zeta} respectively.
\end{theorem}

\begin{proof}
Consider the automorphism of $\text{Pic}^0(\text{Pic}^0(X))$ induced by 
the involution of ${\rm Pic}^0(X)$ defined by
$L\, \longmapsto\, L^*$. It fixes the line bundle $L_0$
defined in \eqref{l0}, because
the above involution $L\, \longmapsto\, L^*$
fixes the point of ${\rm Pic}^0(X)$ corresponding
to the trivial line bundle ${\mathcal O}_X$ on
$X$. Note that $\text{Pic}^0(\text{Pic}^0(X))$ is identified
with ${\rm Pic}^0(X)$ by sending any $\xi\in {\rm Pic}^0(X)$ to $\mathcal{O}_{{\rm Pic}^0(X)}(\xi - \mathcal{O}_X)$. 
This identification commutes with the
involutions. Since the point of ${\rm Pic}^0(X)$
corresponding to ${\mathcal O}_X$ is fixed by the involution,
it follows that
\begin{equation}\label{ui}
{\sigma}_d(g)^*(\bigotimes_{i=1}^d q^*_i L_0)\, =\,
\bigotimes_{i=1}^d q^*_i L_0
\end{equation}
for all $g\, \in\, ({\mathbb Z}/2{\mathbb Z})^d$. This proves the
first statement of the theorem.

Let $z_0\, :=\, ({\mathcal O}_X\, ,\cdots\, ,{\mathcal O}_X)\,
\in\, {\rm Pic}^0(X)^d$ be the point. Note that
${\sigma}_d(g)(z_0)\, =\, z_0$ for all $g\, \in\, ({\mathbb Z}/
2{\mathbb Z})^d$. In view of \eqref{ui},
there is a unique isomorphism
\begin{equation}\label{rho}
\rho\, :\, {\sigma}_d(g)^*\bigotimes_{i=1}^d q^*_i L_0\, 
\longrightarrow\,\bigotimes_{i=1}^d q^*_i L_0
\end{equation}
which coincides with the identity map of the fiber
$(\bigotimes_{i=1}^d q^*_i L_0)_{z_0}$ over the point $z_0$.

Consider the action of $({\mathbb Z}/2\mathbb Z)^d$ on
${\rm Pic}^0(X)^d$ defined by the homomorphism ${\sigma}_d$
in \eqref{in4}. For any $g\, \in\, ({\mathbb Z}/2\mathbb Z)^d$,
there is a canonical lift of the involution
${\sigma}_d(g)$ of ${\rm Pic}^0(X)^d$ to the line bundle
$$
(\bigotimes_{i=1}^d q^*_i L_0)\otimes
{\sigma}_d(g)^*(\bigotimes_{i=1}^d q^*_i L_0)\, .
$$
Using the isomorphism $\rho$ in \eqref{rho}, these lifts
of the involutions ${\sigma}_d(g)$, $g\, \in\, ({\mathbb Z}/2\mathbb 
Z)^d$, together produce a lift of the action of $({\mathbb Z}/2\mathbb
Z)^d$ on ${\rm Pic}^0(X)^d$ 
to the line bundle $$(\bigotimes_{i=1}^d q^*_i L_0)^{\otimes 2}
\, \longrightarrow\, {\rm Pic}^0(X)^d\, .$$

We already noted that the action of 
$\Sigma_d$ on ${\rm Pic}^0(X)^d$ has a natural lift to an action of
$\Sigma_d$ on the line bundle $\bigotimes_{i=1}^d q^*_i L_0$. This
lift to $\bigotimes_{i=1}^d q^*_i L_0$ produces a lift to 
$(\bigotimes_{i=1}^d q^*_i L_0)^{\otimes 2}$
of the action of $\Sigma_d$ on ${\rm Pic}^0(X)^d$. This action
of $\Sigma_d$ on $(\bigotimes_{i=1}^d q^*_i L_0)^{\otimes 2}$
and the action of $({\mathbb Z}/2\mathbb
Z)^d$ on $(\bigotimes_{i=1}^d q^*_i L_0)^{\otimes 2}$ constructed
above together produce
a lift to $(\bigotimes_{i=1}^d q^*_i L_0)^{\otimes 2}$
of the action of $\Gamma_d$ on ${\rm Pic}^0(X)^d$. This proves the
second statement of the theorem.

Let ${\mathcal N}_X(r)$ denote the moduli space of semistable
vector bundles on $X$ of rank $r$ and degree zero. So,
${\mathcal N}_X(r)\,=\, {\rm Sym}^r({\rm Pic}^0(X))\, :=\,
{\rm Pic}^0(X)^r/\Sigma_r$.

Let
$$
\beta\, :\, {\rm Pic}^0(X)^d\, \longrightarrow\, {\mathcal N}_X(r)
$$
be the morphism defined by
\begin{equation}\label{beta2}
(L_1\, ,\cdots \, ,L_d)\, \longmapsto\,
\bigoplus_{i=1}^d (L_i\oplus L^*_i) \oplus \bigoplus_{j=1}^{r-2d}
\eta_j
\end{equation}
(see \eqref{e5}). Let
\begin{equation}\label{ga2}
\gamma\, :\, {\mathcal M}_P\, \longrightarrow\, {\mathcal N}_X(r)
\end{equation}
be the morphism that sends any parabolic vector bundle on
${\mathbb C}{\mathbb P}^1$ to the corresponding vector
bundle on $X$ (see the proof of Proposition \ref{prop1}). Clearly,
\begin{equation}\label{c}
\beta\, :=\, \gamma\circ\phi\, ,
\end{equation}
where $\phi$ is constructed in \eqref{phi}.

Let $\zeta_r$ be the determinant line bundle on the moduli
space ${\mathcal N}_X(r)$. We will quickly recall the
definition/construction of $\zeta_r$. Let
\begin{equation}\label{po}
{\mathcal P} \, \longrightarrow\, X\times{\rm Pic}^0(X)
\end{equation}
be a Poincar\'e line bundle; this means that for each point
$\alpha\, \in\, {\rm Pic}^0(X)$, the restriction
${\mathcal P}\vert_{X\times\{\alpha\}}$ lies in the isomorphism
class of line bundles defined by the point $\alpha$. Let
\begin{equation}\label{pi2}
\pi_2\, :\, X\times{\rm Pic}^0(X)\, \longrightarrow\,{\rm Pic}^0(X)
\end{equation}
be the natural projection. Define the line bundle
$$
{\mathcal L}\, :=\, (\bigwedge^{\rm top} R^0\pi_{2*}{\mathcal P})^*
\otimes \bigwedge^{\rm top} R^1\pi_{2*}{\mathcal P}
\, \longrightarrow\,{\rm Pic}^0(X)\, .
$$
It can be shown that ${\mathcal L}$ is independent of the choice
of the Poincar\'e line bundle $\mathcal P$. To see this note that
any other Poincar\'e bundle is of the form ${\mathcal P}_1\, :=\,
{\mathcal P}\otimes \pi^*_2 
A$, where $A$ is a line bundle on
${\rm Pic}^0(X)$. Form the projection formula,
$$
(\bigwedge^{\rm top} R^0\pi_{2*}{\mathcal P})^*
\otimes \bigwedge^{\rm top} R^1\pi_{2*}{\mathcal P}\,=\,
(\bigwedge^{\rm top} R^0\pi_{2*}{\mathcal P}_1)^*
\otimes (\bigwedge^{\rm top} R^1\pi_{2*}{\mathcal P}_1)\otimes
A^{\otimes \chi}\, ,
$$
where $\chi$ is the Euler characteristic of degree zero line bundles
on $X$. Since $\chi\, =\, 0$, we have ${\mathcal L}\,=\,
(\bigwedge^{\rm top} R^0\pi_{2*}{\mathcal P}_1)^*
\otimes\bigwedge^{\rm top} R^1\pi_{2*}{\mathcal P}_1$.

We will show that ${\mathcal L}$ coincides with $L_0$ defined
in \eqref{l0}. To prove this, we first note that ${\mathcal O}_X$
is the unique line
bundle on $X$ such that $\chi({\mathcal O}_X)\,=\, 0\, \not=\,
H^0(X,\, {\mathcal O}_X)$. From this it follows that the point
of ${\rm Pic}^0(X)$ defined by ${\mathcal O}_X$ is the
canonical theta divisor. This immediately implies that
${\mathcal L}$ is canonically identified with $L_0$.

For each
$i\,\in\, [1\, ,r]$, let $\overline{q}_i$ be the projection of
${\rm Pic}^0(X)^r$ to the $i$--th factor. The line bundle
\begin{equation}\label{f1}
\bigotimes_{i=1}^r\overline{q}^*_i{\mathcal L}\, \longrightarrow
\,{\rm Pic}^0(X)^r
\end{equation}
has a natural action of the group $\Sigma_r$ of permutations of
$\{1\, ,\cdots\, ,r\}$. Using this action, the line bundle
in \eqref{f1}
descends to the quotient ${\rm Sym}^r({\rm Pic}^0(X))$
of ${\rm Pic}^0(X)^r$. This descended line bundle is the
determinant line bundle $\zeta_r$ on ${\mathcal N}_X(r)\,
=\, {\rm Sym}^r({\rm Pic}^0(X))$.

For the map $\gamma$ in \eqref{ga2}, the pullback $\gamma^*\zeta_r$
coincides with the determinant line bundle $\zeta$ on
${\mathcal M}_P$ \cite{BR}, \cite{Bi2}. Therefore, from
\eqref{c} we get an isomorphism
\begin{equation}\label{c1}
\phi^*\zeta\,\stackrel{\sim}{\longrightarrow}\,\beta^*\zeta_r\, .
\end{equation}
Using the fact that each $\eta_j$ in \eqref{beta2} is a fixed
line bundle of order
two, from the construction of $\zeta_r$ described above
it is easy to see that $\beta^*\zeta_r$ has a canonical
lift of the action of $\Gamma_d$ on ${\rm Pic}^0(X)^d$. As noted
earlier, the line bundle $\phi^*\zeta$ is equipped with an
action of $\Gamma_d$, where $\phi$ is constructed
in \eqref{phi}. It is straightforward 
to check that the
isomorphism in \eqref{c1} intertwines the actions of $\Gamma_d$.

For any Poincar\'e line bundle ${\mathcal P} \, \longrightarrow\, 
X\times{\rm Pic}^0(X)$ (see \eqref{po}), the pullback
$$
(\text{Id}_X\times \sigma_1)^*{\mathcal P}^* \, \longrightarrow\,
X\times{\rm Pic}^0(X)
$$
is also a Poincar\'e line bundle, where $\sigma_1$ is the
involution in \eqref{in4} defined by $L\, \longmapsto\, L^*$. 
Therefore,
\begin{equation}\label{f2}
(\bigwedge^{\rm top} R^0\pi_{2*}{\mathcal P}^*)^*
\otimes \bigwedge^{\rm top} (R^1\pi_{2*}{\mathcal P}^*)
\,=\, \sigma^*_1 {\mathcal L}\, ,
\end{equation}
where $\pi_2$ is the projection in \eqref{pi2}.
Since ${\mathcal L}\,=\, L_0$,
\begin{equation}\label{f3}
\sigma^*_1 {\mathcal L}\,=\, \sigma^*_1 L_0\,=\, L_0\, ;
\end{equation}
the last isomorphism follows from the fact that
the point of ${\rm Pic}^0(X)$ corresponding to ${\mathcal O}_X$
is fixed by $\sigma_1$ (see also \eqref{l0}). Combining
\eqref{f2} and \eqref{f3},
\begin{equation}\label{di}
(\bigwedge^{\rm top} R^0\pi_{2*}{\mathcal P}^*)^* \otimes
\bigwedge^{\rm top} (R^1\pi_{2*}{\mathcal P}^*)\,=\, L_0\, .
\end{equation}
Using \eqref{di}, from the constructions of the line bundle
$\zeta_r$ and the morphism $\beta$ in \eqref{beta2} it follows
that 
$$
\beta^*\zeta_r\,=\, (\bigotimes_{i=1}^d q^*_i L_0)^{\otimes 2}\, .
$$
In the second part of the theorem we constructed an action
of the group $\Gamma_d$ on the line bundle
$(\bigotimes_{i=1}^d q^*_i L_0)^{\otimes 2}$. We noted earlier that
$\beta^*\zeta_r$ is equipped with a lift of the action of $\Gamma_d$
on ${\rm Pic}^0(X)^d$. The above isomorphism of $\beta^*\zeta_r$
with $(\bigotimes_{i=1}^d q^*_i L_0)^{\otimes 2}$
is $\Gamma_d$--equivariant. In view of the fact, noted earlier,
that the isomorphism in \eqref{c1} intertwines the actions
of $\Gamma_d$, this completes the proof of the theorem.
\end{proof}

There is a unique translation invariant K\"ahler form
$h_0$ on $\text{Pic}^0(X)$ of total volume one. Let
$$
\omega\, :=\, \sum_{i=1}^d q^*_i h_0
$$
be the K\"ahler form on ${\rm Pic}^0(X)^d$, where $q_i$ is
the projection in \eqref{qi}. As before, the
K\"ahler form on ${\mathcal M}_P$ will be denoted by $\omega_P$.

\begin{proposition}\label{prop2}
For the morphism $\phi$ in \eqref{phi},
$$
\phi^*\omega_P\,=\, 2\omega\, .
$$
\end{proposition}

\begin{proof}
Let $\omega_r$ be the K\"ahler form on the moduli space
${\mathcal N}_X(r)$. For the map $\gamma$ in \eqref{ga2},
\begin{equation}\label{bri}
\gamma^*\omega_r\,=\, \omega_P
\end{equation}
(see \cite{BR}).

It can be shown that
\begin{equation}\label{brj}
\beta^*\omega_r\,=\, 2\omega\, ,
\end{equation}
where $\beta$ is constructed in \eqref{beta2}. To prove
\eqref{brj}, we first recall that $\omega_r$ is constructed
using the unique unitary flat connection on polystable vector
bundles of degree zero over $X$. More precisely, consider the
unique unitary flat connection $\nabla$ on a polystable vector bundle
$$
E\,:=\, \bigoplus_{i=1}^r L_i\, \in\, {\mathcal N}_X(r)
$$
(the flat hermitian metric on $E$ is not unique, but the
flat hermitian connection is unique). Let $\widetilde{\nabla}$ be
the flat connection on $End(E)\,=\, E\otimes E^*$ induced by
$\nabla$. The tangent space $T_E {\mathcal N}_X(r)$ is identified with
\begin{equation}\label{brk}
\bigoplus_{i=1}^r H^1(X, \, End(L_i))\,=\,
H^1(X, \, {\mathcal O}_X)^{\oplus r}\, \subset\,
H^1(X,\, End(E))\, .
\end{equation}
Using the flat unitary structure on $End(E)$, we can represent
elements of $H^1(X,\, End(E))$ by the harmonic forms. This yields
a $L^2$--metric on $H^1(X,\, End(E))$. The restriction of this form
to the subspace $H^1(X, \, {\mathcal O}_X)^{\oplus r}$ in \eqref{brk}
coincides with the K\"ahler form $\omega_r$ on $T_E{\mathcal N}_X(r)$.

The equality in \eqref{brj} follows from the above description of 
$\omega_r$. Note that the factor $2$ in \eqref{brj} appears because
the map $\beta$ constructed in \eqref{beta2} involves both $L_i$
and $L^*_i$, and the involution of $\text{Pic}^0(X)$ defined by
$L\, \longmapsto\, L^*$ preserves the translation invariant K\"ahler 
form $h_0$ on $\text{Pic}^0(X)$.

The proposition follows from \eqref{bri}, \eqref{brj} and
\eqref{c}.
\end{proof}

There is a unique additive complex
Lie group structure on $X$ with $f^{-1}(0)$ as the identity element,
where $f$ is the map in \eqref{e1}.
We fix this Lie group structure on $X$. The identity element
$f^{-1}(0)$ will be denoted by $e$.

There is a natural complex group homomorphism
\begin{equation}\label{e2}
X\, \longrightarrow\, \text{Pic}^0(X)
\end{equation}
defined by $x\, \longmapsto\, {\mathcal O}_X(x-e)$. This isomorphism
will be useful here.

\section{Non-abelian theta functions}

In \cite{FMN2}, non-abelian theta functions on the moduli space of
trivial determinant vector bundles of rank of $n$ on the elliptic curve
$X$ were studied in terms of Weyl anti-invariant distributions in
${\rm SU}(n)$. Let us recall briefly that construction, paying particular
attention to the case $n=2$, which will be especially relevant for
the description of the Hilbert space associated to the quantization
of the moduli space of parabolic bundles $\mathcal{M}_{P}$.

\subsection{${\rm SL}_{n}(\mathbb{C})$ non-abelian theta functions on an elliptic curve}
We start by writing the elliptic curve $X$ in the form: \begin{equation}
X\,=\,\mathbb{C}/(\mathbb{Z}\oplus\tau\mathbb{Z}),\ \mbox{ for some }\tau\in\mathbb{H}\,=\,\{z\in\mathbb{C}:\,\ {\rm Im}(z)>0\}\,.\label{eq:elliptic-curve}\end{equation}
 Let $\mathfrak{h}$ be the Cartan subalgebra of $sl_{n}(\mathbb{C})$
consisting of diagonal matrices of trace zero,
and let $\check{\Lambda}$ denote its coroot lattice. To be concrete,
we identify $sl_{n}(\mathbb{C})$ with the space of traceless $n$
by $n$ complex matrices and $\mathfrak{h}$ with the space of diagonal
matrices of trace zero.

Let $\mathcal{M}_{X}(n)$ be the moduli space semistable vector bundles
$E$ over $X$ of rank $n$ with $\bigwedge^{n}E\,=\,{\mathcal{O}}_{X}$.

Consider the abelian variety\[
M\,=\, X\otimes\check{\Lambda}\,\cong\,\mathfrak{h}/(\check{\Lambda}\oplus\tau\check{\Lambda})\,.\]
 The Weyl group $W$ of $sl_{n}(\C)$, given by the permutations of
$\{1\,,\cdots\,,n\}$, acts naturally on $M$, via its natural action
on $\mathfrak{h}$. As shown in \cite{Lo,La}, the moduli space $\mathcal{M}_{X}(n)$
can be naturally identified with the quotient under this action\[
\mathcal{M}_{X}(n)\,=\, M/W\,\cong\,\mathbb{CP}^{n-1}\,.\]

To consider the quantization of $\mathcal{M}_{X}(n)$, we use the
symplectic form $\omega$ induced from the symplectic structure on
$X$ and the determinant line bundle $L\,\longrightarrow\,
\mathcal{M}_{X}(n)$
whose curvature form coincides with $\omega$ \cite{Qu}.

Let \[
p\,:\,\mathbb{C}/\mathbb{Z}\,\cong\,\mathbb{C}^{*}\,\longrightarrow\,{\rm Pic}^{0}(X)\,\cong\, X\]
 be the projection defined by $z+\mathbb{Z}\,\longmapsto\, z+\mathbb{Z}+\tau\mathbb{Z}$,
where $z\in\mathbb{C}$. The maximal torus of diagonal matrices in
${\rm SL}_{n}(\mathbb{C})$ will be denoted by $T_{\mathbb{C}}$ and
is canonically identified with $\mathfrak{h}/\check{\Lambda}$. Let
\begin{equation}
\label{qiu}
q\,:\,\ {\rm SL}_{n}(\mathbb{C})\,\longrightarrow\,{\rm SL}_{n}(\mathbb{C})/{\rm SL}_{n}(\mathbb{C})\,\cong\, T_{\mathbb{C}}/W
\end{equation}
 be the quotient map for the conjugation action of ${\rm SL}_{n}(\mathbb{C})$
on itself. We have the following commutative diagram:

\begin{equation}
\begin{array}{ccccccc}
{\rm SL}_{n}(\mathbb{C}) & \stackrel{q}{\longrightarrow} & T_{\mathbb{C}}/W & \longleftarrow & T_{\mathbb{C}} & = & \mathfrak{h}/\check{\Lambda}\\
 & & \Big\downarrow & & \Big\downarrow\\
 & & \mathcal{M}_{X}(n) & \longleftarrow & M & \longleftarrow & \mathfrak{h},\end{array}\label{diagram}\end{equation}
 where the composition ${\rm SL}_{n}(\mathbb{C})\to T_{\mathbb{C}}/W\to\mathcal{M}_{X}(n)$
corresponds to the Schottky map described in \cite{FMN2}.

Note that the Verlinde numbers are: \begin{equation}
\dim H^{0}(\mathcal{M}_{X}(n),\, L^{k})=\binom{n-1+k}{k}. 
\label{verlinde}\end{equation}

Let $\mathcal{H}(V)$ be the space of all holomorphic functions on
a complex manifold $V$. Given a level $k$, the subspace of $\mathcal{H}(\mathfrak{h})$
consisting of functions $\theta$ satisfying the identity \begin{equation}
\theta(v+\check{\alpha}+\tau\check{\beta})\,=\,\left(e^{-2\pi i\beta(\upsilon)-\pi i\tau\left\langle \beta,\beta\right\rangle }\right)^{k}\theta(v)\,,\quad\check{\alpha}\,,\check{\beta}\,\in\,\check{\Lambda}\label{eq:quasiperiodic}\end{equation}
 will be relevant for us.

\begin{proposition}[\cite{FMN2}]
The space of non-abelian theta functions $H^{0}(\mathcal{M}_{X}(n),\, L^{k})$
is naturally identified with
\[
\mathcal{H}_{k,n}^{+}\,:=\,\{\theta\in\mathcal{H}(\mathfrak{h})
\,:\,\ \theta\mbox{ satisfies }\eqref{eq:quasiperiodic} \,\,{and}\,\,
\ w\theta=\theta,\ \forall\, w\in W\}\, .
\]
\end{proposition}

We remark that since the quasi-periodicity condition in \eqref{eq:quasiperiodic}
does not depend on the first summand (i.e, $\check{\alpha}$) of the
lattice $\check{\Lambda}\oplus\tau\check{\Lambda}$, the non-abelian
theta functions in the proposition can be also considered to be Weyl
invariant holomorphic functions on $T_{\mathbb{C}}=\mathfrak{h}/\check{\Lambda}$,
or equivalently, Ad--invariant holomorphic functions on ${\rm SL}_{n}(\mathbb{C})$.

Motivated by the Segal--Bargmann--Hall or ``coherent state'' transform for Lie groups,
we will now describe a way of obtaining such Ad--invariant holomorphic
functions on $G\,=\,{\rm SL}_{n}(\mathbb{C})$ starting from Ad--invariant
distributions on the maximal compact subgroup $K\,=\,{\rm SU}(n)$.

Let $\Lambda_{W}^{+}$ denote the set of dominant weights, in one
to one correspondence with irreducible representations $R_{\lambda}$
of $K=SU(n)$. For $x\,\in\, K\,$, the expression \begin{equation}
f\ =\,\sum_{\lambda\in\Lambda_{W}^{+}}\mathrm{tr}(A_{\lambda}R_{\lambda})\,,\label{eq:distribution}\end{equation}
 where $A_{\lambda}\,\in\,\mathrm{End}(R_{\lambda})$ are endomorphism--valued
coefficients, defines a distribution under appropriate growth conditions
on the operator norm of the $A_{\lambda}$ (see \cite{FMN2}).

Let $c_{\lambda}\,\geq\,0$ be the eigenvalue of $-\Delta_{K}$, where
$\Delta_{K}$ is the Laplace-Beltrami operator on $K$ associated
with the Ad-invariant inner product on $\mathfrak{su}_{n}$ for which the roots have squared length 2, on functions
of the form $\mathrm{tr}(A_{\lambda}R_{\lambda}(x))$, $A_{\lambda}\,\in\,\mathrm{End}(R_{\lambda})$. 

Given a positive parameter $t>0$, and $\tau\in{\mathbb H}$, the (generalized) coherent state
transform (CST for short) is given by associating to a distribution
$f$ as in (\ref{eq:distribution}) the holomorphic function on ${\rm SL}_{n}(\mathbb{C})$
\[
C_tf(g)\,:=\,\sum_{\lambda\in\Lambda_{W}^{+}}e^{i\pi t\tau c_{\lambda}}\mathrm{tr}(A_{\lambda}R_{\lambda}(g)).
\]

Recall from \cite{Lo,FMN2} that
non-abelian theta functions on $\mathcal{M}_X(n)$ are more conveniently
described in terms of Weyl anti-invariant theta functions on $\mathfrak{h}$.
Denote by
$\theta_n^-$ the unique (up to scale) $W$-anti-invariant theta function
of level $n$ on $\mathfrak{h}$.

Let now $\rho$ be the Weyl vector given by half the sum of the positive roots and let $\sigma$ be 
the denominator of the Weyl character formula analytically continued to ${\rm SL}_n(\C)$. 
Let $\check \alpha$ be the longest root in $\mathfrak{sl}_{n}(\mathbb{C})$
and let
\[
D_{k,n}\,:=\,\{\lambda\in\Lambda_{W}^{+}\,:\,\left\langle \lambda,\check{\alpha}\right\rangle \leq k\}\]
be the parameter space for integrable representations of the level
$k$ affine Kac-Moody algebra $\widehat{\mathfrak{sl}_{n}}(\mathbb{C})_{k}$.
Note that $\#D_{k,n}\,=\,\binom{n-1+k}{k}$, which equals the Verlinde
number (\ref{verlinde}).

As seen in \cite{FMN2}, there is a ($\tau$-independent) 
finite-dimensional space of Ad--invariant distributions $V_{k,n}$
on ${\rm SU(n)}$ which has an orthonormal basis labelled by the 
elements of $D_{k,n}$, such that the following holds:

\begin{theorem}[\cite{FMN2}] Let $n>2$ and $C^{\infty}({\rm SU}(n))^{'Ad}\supset L^{2}({\rm SU}(n))$ 
denote the space of $Ad$-invariant distributions
on ${\rm SU}(n)$. Restricting the CST to $V_{k,n}$, we obtain 
\[
V_{k,n} \hookrightarrow C({\rm SU}(n))'^{Ad} \stackrel{C_{t}}{\longrightarrow} \mathcal{H}({\rm SL}_n(\C))^{Ad}.
\]
Moreover, the composition of maps
\begin{equation}\label{multi}
\varphi_{k,\tau}\circ C_{\frac{1}{k+n}}: V_{k,n} \longrightarrow H^{0}
(\mathcal{M}_{X}(n),L^{k}) \subset \mathcal{H}({\rm SL}_n(\C))^{Ad}\, ,
\end{equation}
where
\begin{equation}\label{multi2}
\varphi_{k,\tau}(f) = e^{\frac{||\rho||^2}{k+n}\pi i \tau} \frac{\sigma}{\theta^-_{n}} f,
\end{equation}
is an isomorphism; we identify a $W$-invariant theta function on 
$\mathfrak h$ 
with an $Ad$-invariant function on ${\rm SL}_n(\C)$ using 
(\ref{diagram}) and (\ref{eq:quasiperiodic}).
\end{theorem}

\begin{remark} {\rm Note that the map $\varphi_{k,\tau}$ is well defined only on $C_{\frac{1}{k+n}}(V_{k,n})$, since 
these holomorphic functions are divisible by $\theta^-_{n}$ \cite{Lo,FMN2}.}
\end{remark}

Let $\tau_2\,=\, {\rm Im}(\tau)\,>\,0$, and define the hermitian inner 
product on 
$H^{0}(\mathcal{M}_{X}(n),\, L^{k})$ by
\begin{equation}
\left\langle \left\langle F_{1},F_{2}\right\rangle \right\rangle \,:=\,\int_{q^{-1}(\mathfrak{h}_{0})}
\overline{F_{1}}F_{2}|q^* 
\theta^-_n|^2d\nu_{\frac{\tau_2}{k+n}}
\label{eq:inner}\end{equation}
(see \cite{AdPW}),
 where $q$ was defined in equation (\ref{qiu}), $d\nu_{\frac{\tau_2}{k+n}}$ denotes what 
is known as the heat kernel measure
of ${\rm SL}_{n}(\mathbb{C})$, at time $\frac{\tau_2}{k+n}$,
and $\mathfrak{h}_{0}\,\subset\,\mathfrak{h}$ is a fundamental domain
for the action of the semi-direct product $W\ltimes (\check{\Lambda}\oplus\tau\check{\Lambda})$.

\begin{theorem}[\cite{FMN2}]
The map $\varphi_{k,\tau}\circ C_{\frac{1}{k+n}}:V_{k,n}\longrightarrow 
H^0(\mathcal{M}_{X}(n),\, L^{k})$ is a unitary isomorphism.
\end{theorem}

Let us now consider the (slightly different) case $n\,=\,2$, which will be especially
relevant in the next section, and for which the distributions in the
previous theorem can be written in a simple way. For simplicity, we 
will state the result only for even level, which is the case we will need.

The space $V_{2k,2}\,\subset\, C({\rm SU}(2))'^{Ad}$ is the 
$k$-dimensional $\C$-span of the distributions
\begin{equation}
\psi_{j,2k}(x)=\frac{1}{\sigma}\sum_{n\in\Z}(e^{2\pi i(j+2kn)x}-e^{-2\pi i(j+2kn)x})\in C^{\infty}({\rm SU}(2))', 
j=1,\dots,k
\label{psis}
\end{equation}
(see \cite{FMN2}).

Consider the basis of level $2k+4$ theta functions for the elliptic 
curve $X$, namely $\{\theta_{j,2k+4}\}_{0\leq j<2k+4}$, with
\[
\theta_{j,2k+4}(z)=\sum_{m\in\Z}\exp\left(\pi i\frac{\tau}{2k+4}(j+(2k+4)m)^{2}+2\pi i(j+(2k+4)m)z\right),\,\,0\leq j<2k+4.
\]
The Weyl anti-invariant theta function of level $4$ on $X$
is given by $\theta_{4}^{-}=\theta_{1,4}-\theta_{3,4}$.

Recall from \cite{Lo,FMN2} that the space of Weyl invariant theta
functions of level $2k$ on $X$ can be conveniently described in
terms of Weyl anti-invariant theta functions of level $2k+4$, \begin{equation}
\begin{array}{rcc}
H^{0}(X,L_{0}^{2k+4})^{-} & \cong & H^{0}(X,L_{0}^{2k})^{+}\\
\theta^{-} & \longmapsto & \theta^{+}=\theta^{-}/\theta_{4}^{-},\end{array}\end{equation}
 where $\theta_{4}^{-}$ is the (unique up to nonzero multiplicative 
constant)
Weyl anti-invariant theta function of level $4$ on $X$. The bundle
of conformal blocks (of level $k$) over ${\mathcal{M}}_{1}$, which
is associated to the moduli space of semistable rank two vector bundles
with trivial determinant on $X$, has a natural hermitian structure
which is easily expressed in terms of theta functions in 
$H^{0}(X,L_{0}^{2k+4})^{-}$ as described above \cite{AdPW,FMN2}.

\begin{theorem} \label{nequal2}
The composition of maps 
\begin{eqnarray}\nonumber
V_{2k,2} \stackrel{C_{\frac{1}{k+2}}}{\longrightarrow} C_{\frac{1}{k+2}}(V_{2k,2}) \stackrel{\varphi_{k,\tau}}
{\longrightarrow} H^{0}(\mathcal{M}_{X}(2),L^{2k}) 
\subset \mathcal{H}({\rm SL}_2(\C))^{Ad}
\end{eqnarray}
where $\varphi_{k,\tau}(f) = e^{\frac{1}{2k+4}\pi i \tau} \frac{\sigma}{\theta^-_{4}} f$,
is an isomorphism. The image of the natural basis 
$\{\psi_{j,2k}\}_{j=1,\dots,k}$ is given by 
\begin{equation}
\{\vartheta_{j,2k}\}_{j=1,\cdots,k}\, ,\label{gaspar}\end{equation}
 where 
$\vartheta_{j,2k}=(\theta_{j,2k+4}-\theta_{2k+4-j,2k+4})/\theta_{4}^{-}$
\cite{FMN2}.
\end{theorem}

\subsection{Non-abelian theta functions on $\mathcal{M}_{P}$}

Let $\check \Lambda$ be the coroot lattice of $sl_2(\C)$, and let $\h$,
as before, be the Cartan subalgebra. The abelian variety from 
the previous subsection is now $M=X\otimes \check \Lambda \cong X$.

{}From Proposition \ref{prop1}, 
we have 
$$
\mathcal{M}_{P} \cong {\rm Pic}^0(X)^d/\Gamma_d \cong M^d/\Gamma_d\, ,
$$
where $\Gamma_d = \Z_2^d \rtimes \Sigma_d$.
We have the 
isomorphism of abelian varieties
$$
M^d = \left(\h/\check \Lambda\oplus \tau \check \Lambda\right)^d . 
$$

Let $p:M^d\longrightarrow M^d/\Gamma_d$ be the natural projection. From 
above, the pull-back of the determinant line bundle by $p$ gives 
the line bundle $\bigotimes_{i=1}^d (q_i^* L_0)^2$ over $M^d$. Therefore, 
non-abelian theta functions of level $k$ on
$\mathcal{M}_{P}$ will be described by $\Gamma_d$ invariant products of 
level $2k$ theta functions on each of the 
factors $\h/\check \Lambda\oplus \tau \check \Lambda$.

The analog of diagram (\ref{diagram}) is now

\begin{equation*}
\begin{array}{ccccccc} 
{\rm SL}_{2}(\mathbb{C})^d & \longrightarrow & T_{\mathbb{C}}^d/W^d & \longleftarrow
& T_{\mathbb{C}}^d & = & (\mathfrak{h}/\check{\Lambda})^d\\ 
 & & \Big\downarrow & & \Big\downarrow & & \\
\mathcal{M}_P & \stackrel{/\Sigma_d}{\longleftarrow} & \mathcal{M}_{X}(2)^d & \longleftarrow & M^d & \longleftarrow &
\mathfrak{h}^d.\end{array}\end{equation*}

Let $K$ be a compact Lie group. The CST on $K^d$ equipped with the product 
metric can be applied to $\Sigma_d$-invariant functions. Since the averaged heat kernel measures are 
the product of the $d$ measures on each of the factors of $K$, we have the following commutative diagram
$$
\begin{matrix}
(C^\infty(K^d)')^{\Sigma_d} & {\hookrightarrow} & C^\infty(K^d)'\\
~\Big \downarrow C_t^{\otimes d} & & ~~\Big\downarrow C_t^{\otimes d}\\
 \mathcal{H}(G)^{\Sigma_d} & {\hookrightarrow} & \mathcal{H}(G).
\end{matrix}
$$
where the CST for $K^d$ is given by
$$
C_t^{\otimes d} = C_t \otimes \cdots \otimes C_t\, ,
$$
in terms of the CST $C_t$ for $K$.
 
\begin{definition}\label{vk} Let $V_{k}\subset C^\infty (({\rm SU}(2)^d)')^{\Sigma_d}$ 
be the vector space with basis \break
$\{\Psi_{J,k}\}_{J=\{j_1,\cdots,j_d\},\, 1\leq j_i \leq k,}$, where
$$
\Psi_{J,k} = \sum_{\sigma\in \Sigma_d} \psi_{j_{\sigma_1},2k}\otimes \cdots \otimes 
\psi_{j_{\sigma_d},2k}\, ,
$$
and the distributions $\psi_{j,2k}\in C^\infty({\rm SU}(2))'$ are given in (\ref{psis}).
\end{definition}

Let $\varphi_{k,\tau}^{\otimes d}=\varphi_{k,\tau}\otimes \cdots \otimes \varphi_{k,\tau}$ be defined on 
$C_{\frac{1}{k+2}}^{\otimes d}(V_k)\subset {\mathcal H}({\rm SL}_2(\C)^d)$, where $\varphi_{k,\tau}$ was defined in 
(\ref{multi2}).

\begin{theorem} \label{main}
The CST $C_{\frac{1}{k+2}}^{\otimes d}$ establishes an isomorphism between the space $V_k$ of distributions on 
${\rm SU(2)}^d$ and 
the space of non-abelian theta functions $H^0(\mathcal{M}_{P},\xi^k)$ of 
level $k$, meaning the map
\begin{eqnarray}\nonumber
\varphi_{k,\tau}^{\otimes d}\circ C_{\frac{1}{k+2}}^{\otimes 
d}:V_k\longrightarrow H^0(\mathcal{M}_{P},\xi^k)
\end{eqnarray}
is an isomorphism.
The image of the 
natural basis (Definition \ref{vk}) is given by \break
$\{\phi_{J,k}\}_{J=\{j_1,\cdots,j_d\},\, 1\leq j_i \leq k}$, where
$$
\phi_{J,k}(z_1,\cdots,z_d) = \sum_{\sigma\in \Sigma_d} \vartheta_{j_{\sigma_1},2k}(z_1)\cdots 
\vartheta_{j_{\sigma_d},2k}(z_d).
$$
\end{theorem}

\begin{proof}
Recall that the determinant line bundle $\xi$ on the moduli space of 
parabolic bundle ${\mathcal M}_P$, satisfies 
$\xi \cong \otimes_{i=1}^d (q_i^* L_0)^2$, where $q_i:M^d \to M$ 
is the projection on the $i$th factor. Therefore, 
elements in $H^0(\mathcal{M}_{P},\xi^k)$ are given by $\Gamma_d$ invariant theta functions of level $2k$ 
on $M^d$. From Theorem \ref{nequal2} it follows that applying 
$\varphi_{k,\tau}^{\otimes d}\circ 
C_{\frac{1}{k+2}}^{\otimes d}$ to $\Psi_{J,k}$ we get $\phi_{J,k}$.
\end{proof}

\begin{remark}{\rm
We see that the Verlinde number is equal to the dimension of the space of degree $k$ polynomials in $d$ variables, 
which is consistent with $\xi \cong {\mathcal O}(1)$ on ${\mathcal M}_P \cong \P^d.$}
\end{remark}

\begin{remark}
{\rm Non-abelian theta functions $H^0(\mathcal{M}_{P},\xi^k)$ of level 
$k$ can therefore be described as the image, by a
coherent state transform, of the finite dimensional space of 
distributions on the compact group ${\rm SU}(2)^d$. In particular, 
following \cite{FMN1,FMN2}, in this case the CST can also be interpreted as the parallel transport of a
unitary connection on the bundle of conformal blocks over ${\mathcal M}_{0,4}$.}
\end{remark}

\section*{Acknowledgements}
The authors were partially supported by the 
Center for Mathematical Analysis, Geometry and Dynamical Systems, 
IST, Portugal. The first author wishes to thank
Instituto Superior T\'ecnico, where the work was carried out,
for its hospitality; his visit to IST was funded by the FCT project
PTDC/MAT/099275/2008.



\begin{thebibliography}{AAAAA}

\bibitem[AdPW]{AdPW} S. Axelrod, S. Della Pietra and E. Witten, 
Geometric quantization of Chern-Simons theory, \textit{Jour. Diff.
Geom.} \textbf{33} (1991), 787--902.

\bibitem[At1]{At1} M. F. Atiyah, On the Krull--Schmidt theorem
with application to sheaves, \textit{Bull. Soc. Math. Fr.}
\textbf{84} (1956), 307--317.

\bibitem[At2]{At2} M. F. Atiyah, Vector bundles over an
elliptic curve, \textit{Proc. London Math. Soc.} \textbf{7}
(1957), 412--452.

\bibitem[BMN]{BMN} T. Baier, J. Mour\~ao and J. P. Nunes, Quantization of abelian varieties: 
distributional sections and the transition from K\"ahler to real 
polarizations, \textit{Jour. Funct. Anal.} \textbf{258} (2010), 3388--3412.

\bibitem[BBN]{BBN} V. Balaji, I. Biswas and D. S. Nagaraj,
Principal bundles over projective manifolds with parabolic
structure over a divisor, \textit{Tohoku Math. Jour.} \textbf{53}
(2001), 337--367.

\bibitem[Bi1]{Bi1} I. Biswas, Parabolic bundles as orbifold
bundles, \textit{Duke Math. Jour.} \textbf{88} (1997), 305--325.

\bibitem[Bi2]{Bi2} I. Biswas, Determinant
line bundle on moduli space of parabolic bundles, \textit{Ann.
Global Anal. Geom.} \textbf{40} (2011), 85--94.

\bibitem[BR]{BR} I. Biswas and N. Raghavendra, Determinants of
parabolic bundles on Riemann surfaces, \textit{Proc. Indian
Acad. Sci. (Math. Sci.)} \textbf{103} (1993), 41--71.

\bibitem[FMN1]{FMN1} C. Florentino, J. Mour\~ao and J. P. Nunes, 
Coherent state transforms and abelian varieties,
\textit{Jour. Funct. Anal.} \textbf{192} (2002), 410--424.

\bibitem[FMN2]{FMN2} C. Florentino, J. Mour\~ao and J. P. Nunes, 
Coherent state transforms and vector bundles on elliptic curves, 
\textit{Jour. Funct. Anal.} \textbf{204} (2003), 355--398.

\bibitem[Go]{G} W. M. Goldman, The symplectic nature of fundamental groups
of surfaces, \textit{Adv. Math.} \textbf{54} (1984), 200--225.

\bibitem[Ha]{Ha} B. Hall, The Segal-Bargmann coherent state transform for 
compact Lie groups, \textit{Jour. Funct. Anal.} \textbf{122} (1994), 
103--151.

\bibitem[Hi]{Hi}N. J. Hitchin, Flat connections and geometric 
quantization, \textit{Comm. Math. Phys.} \textbf{131} (1990), 
347--380.

\bibitem[La]{La} Y. Laszlo, About $G$-bundles over elliptic curves, 
\textit{Ann. Inst. Fourier} \textbf{48} (1998), 413--424.

\bibitem[Lo]{Lo} E. Looijenga, Root systems and elliptic
curves, \textit{Invent. Math.} \textbf{38} (1976), 17--32.

\bibitem[MY]{MY} M. Maruyama and K. Yokogawa, Moduli of parabolic
stable sheaves, \textit{Math. Ann.} \textbf{293} (1992),
77--99.

\bibitem[MS]{MS} V. B. Mehta and C. S. Seshadri, Moduli 
of vector bundles on curves with parabolic structures,
\textit{Math. Ann.} \textbf{248} (1980), 205--239.

\bibitem[NaSe]{NaSe} M. S. Narasimhan and C. S. Seshadri, Stable and 
unitary vector bundles on a compact Riemann surface, {\it Ann. of
Math.} {\bf 82} (1965), 540--567.

\bibitem[Qu]{Qu} D. G. Quillen, Determinants of Cauchy-Riemann
operators on Riemann surfaces, \textit{Funct. Anal. Appl.}
\textbf{19} (1985), 37--41.

\bibitem[TZ]{TZ} L. A. Takhtajan and P. Zograf, The first Chern
form on moduli of parabolic bundles, \textit{Math. Ann.}
\textbf{341} (2008), 113--135.

\end{thebibliography}
\end{document}